\documentclass[runningheads,envcountsame]{llncs}
\usepackage[T1]{fontenc}
\usepackage{graphicx}
\usepackage{hyperref}
\usepackage{color}

\usepackage{amsmath}
\usepackage{amssymb}

\def\user{noname}

\usepackage{tikz}

\begin{document}



\def\AA{{\mathcal A}}
\def\BB{{\mathcal B}}
\def\CC{{\mathcal C}}
\def\DD{{\mathcal D}}
\def\EE{{\mathcal E}}
\def\FF{{\mathcal F}}
\def\GG{{\mathcal G}}
\def\HH{{\mathcal H}}
\def\II{{\mathcal I}}
\def\JJ{{\mathcal J}}
\def\KK{{\mathcal K}}
\def\LL{{\mathcal L}}
\def\MM{{\mathcal M}}
\def\NN{{\mathcal N}}
\def\OO{{\mathcal O}}
\def\PP{{\mathcal P}}
\def\QQ{{\mathcal Q}}
\def\RR{{\mathcal R}}
\def\SS{{\mathcal S}}
\def\TT{{\mathcal T}}
\def\UU{{\mathcal U}}
\def\VV{{\mathcal V}}
\def\WW{{\mathcal W}}
\def\XX{{\mathcal X}}
\def\YY{{\mathcal Y}}
\def\ZZ{{\mathcal Z}}


\def\bA{{\mathbf A}}
\def\bB{{\mathbf B}}
\def\bC{{\mathbf C}}
\def\bD{{\mathbf D}}
\def\bE{{\mathbf E}}
\def\bF{{\mathbf F}}
\def\bG{{\mathbf G}}
\def\bH{{\mathbf H}}
\def\bI{{\mathbf I}}
\def\bJ{{\mathbf J}}
\def\bK{{\mathbf K}}
\def\bL{{\mathbf L}}
\def\bM{{\mathbf M}}
\def\bN{{\mathbf N}}
\def\bO{{\mathbf O}}
\def\bP{{\mathbf P}}
\def\bQ{{\mathbf Q}}
\def\bR{{\mathbf R}}
\def\bS{{\mathbf S}}
\def\bT{{\mathbf T}}
\def\bU{{\mathbf U}}
\def\bV{{\mathbf V}}
\def\bW{{\mathbf W}}
\def\bX{{\mathbf X}}
\def\bY{{\mathbf Y}}
\def\bZ{{\mathbf Z}}


\def\IB{{\mathbb{B}}}
\def\IC{{\mathbb{C}}}
\def\IF{{\mathbb{F}}}
\def\IN{{\mathbb{N}}}
\def\IP{{\mathbb{P}}}
\def\IQ{{\mathbb{Q}}}
\def\IR{{\mathbb{R}}}
\def\IS{{\mathbb{S}}}
\def\IT{{\mathbb{T}}}
\def\IZ{{\mathbb{Z}}}

\def\IIB{{\mathbb{\mathbf B}}}
\def\IIC{{\mathbb{\mathbf C}}}
\def\IIN{{\mathbb{\mathbf N}}}
\def\IIQ{{\mathbb{\mathbf Q}}}
\def\IIR{{\mathbb{\mathbf R}}}
\def\IIZ{{\mathbb{\mathbf Z}}}


\def\ELSE{\quad\mbox{else}\quad}
\def\WITH{\quad\mbox{with}\quad}
\def\FOR{\quad\mbox{for}\quad}
\def\AND{\;\mbox{and}\;}
\def\OR{\;\mbox{or}\;}

\def\To{\longrightarrow}
\def\TO{\Longrightarrow}
\def\In{\subseteq}
\def\sm{\setminus}
\def\Inneq{\In_{\!\!\!\!/}}
\def\dmin{\mathop{\dot{-}}}
\def\splus{\oplus}
\def\SEQ{\triangle}
\def\DIV{\uparrow}
\def\INV{\leftrightarrow}
\def\SET{\Diamond}

\def\kto{\equiv\!\equiv\!>}
\def\kin{\subset\!\subset}
\def\pto{\leadsto}
\def\into{\hookrightarrow}
\def\onto{\to\!\!\!\!\!\to}
\def\prefix{\sqsubseteq}
\def\rel{\leftrightarrow}
\def\mto{\rightrightarrows}

\def\B{{\mathsf{{B}}}}
\def\D{{\mathsf{{D}}}}
\def\G{{\mathsf{{G}}}}
\def\E{{\mathsf{{E}}}}
\def\J{{\mathsf{{J}}}}
\def\K{{\mathsf{{K}}}}
\def\L{{\mathsf{{L}}}}
\def\R{{\mathsf{{R}}}}
\def\T{{\mathsf{{T}}}}
\def\U{{\mathsf{{U}}}}
\def\W{{\mathsf{{W}}}}
\def\Z{{\mathsf{{Z}}}}
\def\w{{\mathsf{{w}}}}
\def\HP{{\mathsf{{H}}}}
\def\C{{\mathsf{{C}}}}
\def\Tot{{\mathsf{{Tot}}}}
\def\Fin{{\mathsf{{Fin}}}}
\def\Cof{{\mathsf{{Cof}}}}
\def\Cor{{\mathsf{{Cor}}}}
\def\Equ{{\mathsf{{Equ}}}}
\def\Com{{\mathsf{{Com}}}}
\def\Inf{{\mathsf{{Inf}}}}

\def\Tr{{\mathrm{Tr}}}
\def\Sierp{{\mathrm Sierpi{\'n}ski}}
\def\psisierp{{\psi^{\mbox{\scriptsize\Sierp}}}}
\def\cl{{\mathrm{{cl}}}}
\def\Haus{{\mathrm{{H}}}}
\def\Ls{{\mathrm{{Ls}}}}
\def\Li{{\mathrm{{Li}}}}

\def\CL{\mathsf{CL}}
\def\ACC{\mathsf{ACC}}
\def\DNC{\mathsf{DNC}}
\def\ATR{\mathsf{ATR}}
\def\LPO{\mathsf{LPO}}
\def\LLPO{\mathsf{LLPO}}
\def\WKL{\mathsf{WKL}}
\def\RCA{\mathsf{RCA}}
\def\ACA{\mathsf{ACA}}
\def\SEP{\mathsf{SEP}}
\def\BCT{\mathsf{BCT}}
\def\IVT{\mathsf{IVT}}
\def\IMT{\mathsf{IMT}}
\def\OMT{\mathsf{OMT}}
\def\CGT{\mathsf{CGT}}
\def\UBT{\mathsf{UBT}}
\def\BWT{\mathsf{BWT}}
\def\HBT{\mathsf{HBT}}
\def\BFT{\mathsf{BFT}}
\def\FPT{\mathsf{FPT}}
\def\WAT{\mathsf{WAT}}
\def\LIN{\mathsf{LIN}}
\def\B{\mathsf{B}}
\def\BF{\mathsf{B_\mathsf{F}}}
\def\BI{\mathsf{B_\mathsf{I}}}
\def\C{\mathsf{C}}
\def\CF{\mathsf{C_\mathsf{F}}}
\def\CN{\mathsf{C_{\IN}}}
\def\CI{\mathsf{C_\mathsf{I}}}
\def\CK{\mathsf{C_\mathsf{K}}}
\def\CA{\mathsf{C_\mathsf{A}}}
\def\WPO{\mathsf{WPO}}
\def\WLPO{\mathsf{WLPO}}
\def\MP{\mathsf{MP}}
\def\BD{\mathsf{BD}}
\def\Fix{\mathsf{Fix}}
\def\Mod{\mathsf{Mod}}

\def\s{\mathrm{s}}
\def\r{\mathrm{r}}
\def\w{\mathsf{w}}

\def\leqm{\mathop{\leq_{\mathrm{m}}}}
\def\equivm{\mathop{\equiv_{\mathrm{m}}}}
\def\leqT{\mathop{\leq_{\mathrm{T}}}}
\def\lT{\mathop{<_{\mathrm{T}}}}
\def\nleqT{\mathop{\not\leq_{\mathrm{T}}}}
\def\equivT{\mathop{\equiv_{\mathrm{T}}}}
\def\nequivT{\mathop{\not\equiv_{\mathrm{T}}}}
\def\leqwtt{\mathop{\leq_{\mathrm{wtt}}}}
\def\equiPT{\mathop{\equiv_{\P\mathrm{T}}}}
\def\leqW{\mathop{\leq_{\mathrm{W}}}}
\def\equivW{\mathop{\equiv_{\mathrm{W}}}}
\def\leqtW{\mathop{\leq_{\mathrm{tW}}}}
\def\leqSW{\mathop{\leq_{\mathrm{sW}}}}
\def\equivSW{\mathop{\equiv_{\mathrm{sW}}}}
\def\leqPW{\mathop{\leq_{\widehat{\mathrm{W}}}}}
\def\equivPW{\mathop{\equiv_{\widehat{\mathrm{W}}}}}
\def\leqFPW{\mathop{\leq_{\mathrm{W}^*}}}
\def\equivFPW{\mathop{\equiv_{\mathrm{W}^*}}}
\def\leqWW{\mathop{\leq_{\overline{\mathrm{W}}}}}
\def\nleqW{\mathop{\not\leq_{\mathrm{W}}}}
\def\nleqSW{\mathop{\not\leq_{\mathrm{sW}}}}
\def\lW{\mathop{<_{\mathrm{W}}}}
\def\lSW{\mathop{<_{\mathrm{sW}}}}
\def\nW{\mathop{|_{\mathrm{W}}}}
\def\nSW{\mathop{|_{\mathrm{sW}}}}
\def\leqt{\mathop{\leq_{\mathrm{t}}}}
\def\equivt{\mathop{\equiv_{\mathrm{t}}}}
\def\leqtop{\mathop{\leq_\mathrm{t}}}
\def\equivtop{\mathop{\equiv_\mathrm{t}}}

\def\bigtimes{\mathop{\mathsf{X}}}

\def\leqm{\mathop{\leq_{\mathrm{m}}}}
\def\equivm{\mathop{\equiv_{\mathrm{m}}}}
\def\leqT{\mathop{\leq_{\mathrm{T}}}}
\def\leqM{\mathop{\leq_{\mathrm{M}}}}
\def\equivT{\mathop{\equiv_{\mathrm{T}}}}
\def\equiPT{\mathop{\equiv_{\P\mathrm{T}}}}
\def\leqW{\mathop{\leq_{\mathrm{W}}}}
\def\equivW{\mathop{\equiv_{\mathrm{W}}}}
\def\nequivW{\mathop{\not\equiv_{\mathrm{W}}}}
\def\leqSW{\mathop{\leq_{\mathrm{sW}}}}
\def\equivSW{\mathop{\equiv_{\mathrm{sW}}}}
\def\leqPW{\mathop{\leq_{\widehat{\mathrm{W}}}}}
\def\equivPW{\mathop{\equiv_{\widehat{\mathrm{W}}}}}
\def\nleqW{\mathop{\not\leq_{\mathrm{W}}}}
\def\nleqSW{\mathop{\not\leq_{\mathrm{sW}}}}
\def\lW{\mathop{<_{\mathrm{W}}}}
\def\lSW{\mathop{<_{\mathrm{sW}}}}
\def\nW{\mathop{|_{\mathrm{W}}}}
\def\nSW{\mathop{|_{\mathrm{sW}}}}

\def\botW{\mathbf{0}}
\def\midW{\mathbf{1}}
\def\topW{\mathbf{\infty}}

\def\pol{{\leq_{\mathrm{pol}}}}
\def\rem{{\mathop{\mathrm{rm}}}}

\def\cc{{\mathrm{c}}}
\def\d{{\,\mathrm{d}}}
\def\e{{\mathrm{e}}}
\def\ii{{\mathrm{i}}}

\def\Cf{C\!f}
\def\id{{\mathrm{id}}}
\def\pr{{\mathrm{pr}}}
\def\inj{{\mathrm{inj}}}
\def\cf{{\mathrm{cf}}}
\def\dom{{\mathrm{dom}}}
\def\range{{\mathrm{range}}}
\def\graph{{\mathrm{graph}}}
\def\Graph{{\mathrm{Graph}}}
\def\epi{{\mathrm{epi}}}
\def\hypo{{\mathrm{hypo}}}
\def\Lim{{\mathrm{Lim}}}
\def\diam{{\mathrm{diam}}}
\def\dist{{\mathrm{dist}}}
\def\supp{{\mathrm{supp}}}
\def\union{{\mathrm{union}}}
\def\fiber{{\mathrm{fiber}}}
\def\ev{{\mathrm{ev}}}
\def\mod{{\mathrm{mod}}}
\def\sat{{\mathrm{sat}}}
\def\seq{{\mathrm{seq}}}
\def\lev{{\mathrm{lev}}}
\def\mind{{\mathrm{mind}}}
\def\arccot{{\mathrm{arccot}}}
\def\cl{{\mathrm{cl}}}
\def\span{{\mathrm{span}}}

\def\Add{{\mathrm{Add}}}
\def\Mul{{\mathrm{Mul}}}
\def\SMul{{\mathrm{SMul}}}
\def\Neg{{\mathrm{Neg}}}
\def\Inv{{\mathrm{Inv}}}
\def\Ord{{\mathrm{Ord}}}
\def\Sqrt{{\mathrm{Sqrt}}}
\def\Re{{\mathrm{Re}}}
\def\Im{{\mathrm{Im}}}
\def\Sup{{\mathrm{Sup}}}

\def\LSC{{\mathcal LSC}}
\def\USC{{\mathcal USC}}

\def\CE{{\mathcal{E}}}
\def\Pref{{\mathrm{Pref}}}

\def\Baire{\IN^\IN}

\def\TRUE{{\mathrm{TRUE}}}
\def\FALSE{{\mathrm{FALSE}}}

\def\co{{\mathrm{co}}}

\def\BBB{{\tt B}}

\newcommand{\SO}[1]{{{\mathbf\Sigma}^0_{#1}}}
\newcommand{\SI}[1]{{{\mathbf\Sigma}^1_{#1}}}
\newcommand{\PO}[1]{{{\mathbf\Pi}^0_{#1}}}
\newcommand{\PI}[1]{{{\mathbf\Pi}^1_{#1}}}
\newcommand{\DO}[1]{{{\mathbf\Delta}^0_{#1}}}
\newcommand{\DI}[1]{{{\mathbf\Delta}^1_{#1}}}
\newcommand{\sO}[1]{{\Sigma^0_{#1}}}
\newcommand{\sI}[1]{{\Sigma^1_{#1}}}
\newcommand{\pO}[1]{{\Pi^0_{#1}}}
\newcommand{\pI}[1]{{\Pi^1_{#1}}}
\newcommand{\dO}[1]{{{\Delta}^0_{#1}}}
\newcommand{\dI}[1]{{{\Delta}^1_{#1}}}
\newcommand{\sP}[1]{{\Sigma^\P_{#1}}}
\newcommand{\pP}[1]{{\Pi^\P_{#1}}}
\newcommand{\dP}[1]{{{\Delta}^\P_{#1}}}
\newcommand{\sE}[1]{{\Sigma^{-1}_{#1}}}
\newcommand{\pE}[1]{{\Pi^{-1}_{#1}}}
\newcommand{\dE}[1]{{\Delta^{-1}_{#1}}}

\newcommand{\dBar}[1]{{\overline{\overline{#1}}}}

\def\QED{$\hspace*{\fill}\Box$}
\def\rand#1{\marginpar{\rule[-#1 mm]{1mm}{#1mm}}}

\def\BL{\BB}


\newcommand{\bra}[1]{\langle#1|}
\newcommand{\ket}[1]{|#1\rangle}
\newcommand{\braket}[2]{\langle#1|#2\rangle}

\newcommand{\ind}[1]{{\em #1}\index{#1}}
\newcommand{\mathbox}[1]{\[\fbox{\rule[-4mm]{0cm}{1cm}$\quad#1$\quad}\]}


\newenvironment{eqcase}{\left\{\begin{array}{lcl}}{\end{array}\right.}

\def\IVP{\mathsf{IVP}}
\def\sign{\mathrm{sign}}

\title{Computability of the\\ Hahn-Banach Theorem Revisited}
%
%
\author{Vasco Brattka\inst{1,2}\orcidID{0000-0003-4664-2183} \and\\
Christopher Sorg\inst{1}\orcidID{0009-0003-9684-6966}}

\authorrunning{V.\ Brattka and C.\ Sorg}
%
\institute{Fakult\"at f\"ur Informatik, Universit\"at der Bundeswehr M\"unchen, Werner-Heisenberg-Weg 39, 85577 Neubiberg, Germany \and
Department of Mathematics and Applied Mathematics, University of Cape Town, Private Bag X3, Rondebosch 7701, South Africa\\
\email{Vasco.Brattka@cca-net.de}\\
\email{chr.sorg@unibw.de}}
\maketitle              
\begin{abstract}
Computational properties of the Hahn-Banach theorem have been studied in
computable, constructive and reverse mathematics and in all these approaches
the theorem is equivalent to weak K\H{o}nig's lemma. Gherardi and Marcone
proved that this is also true in the uniform sense of Weihrauch complexity. However, their result
requires the underlying space to be variable. We prove that the Hahn-Banach
theorem attains its full complexity already for the Banach space $\ell^1$.
We also prove that the one-step Hahn-Banach theorem for this space
is Weihrauch equivalent to the intermediate value theorem. This also yields
a new and very simple proof of the reduction of the Hahn-Banach theorem
to weak K\H{o}nig's lemma using infinite products.
Finally, we show that the Hahn-Banach theorem for $\ell^1$ in the two-dimensional
case is Weihrauch equivalent to the lesser limited principle of omniscience.
 
\end{abstract}
\section{Introduction}

The Hahn-Banach theorem ($\HBT$) is one of the core theorems of functional analysis.
Here and in the following we consider only normed spaces over the field $\IR$
and all computational results are for separable spaces. By $\|f\|:=\sup_{\|x\|\leq1}|f(x)|$ we denote the {\em supremum norm} of a functional $f$ in a normed space $X$.

\begin{theorem}[Hahn-Banach]
Let $X$ be a normed space with a linear subspace $A\In X$. 
Then every linear bounded functional $f:A\to\IR$ has a linear extension $g:X\to\IR$
with $\|g\|=\|f\|$.
\end{theorem}

A good survey on the history of the Hahn-Banach theorem 
is provided by Narici and Beckenstein~\cite{NB11}.
The Hahn-Banach theorem has been studied in computable analysis by Metakides, Nerode and Shore \cite{MN82,MNS85}, in reverse mathematics by Brown and Simpson~\cite{BS86a,Sim09}, and in reverse constructive analysis
by Ishihara~\cite{Ish90,DI21}.
In all these approaches the theorem turns out to be equivalent
to weak K\H{o}nig's lemma ($\WKL$). Finally, Gherardi and Marcone~\cite{GM09}
proved that this is also true in the uniform sense of Weihrauch complexity~\cite{BGP21}
(below we will define Weihrauch equivalence $\equivW$ and other relevant notions used
in the introduction).

\begin{theorem}[Gherardi and Marcone 2009]
\label{thm:GM09}
$\HBT\equivW\WKL$.
\end{theorem}

Using a Kleene tree, one obtains a result of Metakides, Nerode and Shore~\cite{MNS85}
as a corollary of this classification.

\begin{corollary}[Metakides, Nerode and Shore 1985]
\label{cor:MNS85}
There exists a computable Banach space $X$ with a computably separable
subspace $A\In X$ and a computable functional $f:A\to\IR$ with $\|f\|=1$,
such that $f$ has no computable linear extension $g:X\to\IR$ with $\|g\|=1$.
\end{corollary}

Both aforementioned results work with 
the construction of a variable Banach space $X$ that depends
on the respective instance of weak K\H{o}nig's lemma.
We investigate the question whether there is a fixed computable Banach space $X$
for which the Hahn-Banach theorem $\HBT_X$ is 
Weihrauch equivalent to weak K\H{o}nig's lemma.

Every singlevalued problem that maps to a computable metric spaces
and is Weihrauch reducible to $\WKL$ is already computable~\cite[Theorem~8.8]{BG11}.
Hence, Theorem~\ref{thm:GM09}
implies that $\HBT_X$ is computable if the computable normed space $X$
has the property that all linear bounded functionals on subspaces of $X$
have unique norm preserving extensions (this was noted also in~\cite{Bra08b}).

Due to a classical result of Taylor and Foguel~\cite[Theorem~16.4.8]{NB11}, 
it is known that extensions are uniquely determined for a space $X$
(for all subspaces and functionals) if and only if 
the dual space of $X$ is strictly convex. Hence we obtain the following.

\begin{corollary}
$\HBT_X$ is computable for every computable normed space $X$ with
a strictly convex dual space.
\end{corollary}

We recall that the normed spaces $\ell^p(I):=\{x\in\IR^I:\|x\|_p<\infty\}$
over countable sets $I$ are defined for $1\leq p\leq\infty$ by
\[\|(x_i)_{i\in I}\|_p:=\left(\sum_{i\in I}|x_i|^p\right)^\frac{1}{p}\text{ if $p<\infty$ and }\|(x_i)_{i\in I}\|_\infty:=\sup_{i\in I}|x_i|.\]
For $I=\IN$ we briefly write $\ell^p$ and for $I=\{0,1\}$ we write $\ell^p_2$.
The spaces $\ell^p$ for $p$ with $1<p<\infty$ are known
to be strictly convex~\cite[Exercise~16.201]{NB11} and hence so are their dual spaces.
Thus, in a certain sense the space $\ell^1$ is a simple example
of an infinite-dimensional computable normed space whose dual $\ell^\infty$ is not strictly convex. Indeed the Hahn-Banach Theorem already exhibits its full complexity 
on $\ell^1$. One of our main results is the following.

\begin{theorem}
\label{thm:ell1}
	$\HBT_{\ell^1}\equivW\WKL$.
\end{theorem}

We are going to prove this result in Section~\ref{sec:ell1}.
In Section~\ref{sec:located} we strengthen this result
and we show that it does not make the problem $\HBT_{\ell^1}$ simpler,
if we provide more information on the subspace $A$ using its distance function.

Historically, Helly first proved the Hahn-Banach theorem for specific 
normed spaces $X$ inductively, starting from the following
one-step version of the theorem~\cite[Theorem~7.3.1]{NB11}.

\begin{theorem}[One-step Hahn-Banach theorem]
\label{thm:one-step}
Let $X$ be a normed space with a linear subspace $A\In X$ and $x\in X$.
Then every linear functional $f:A\to\IR$ with $\|f\|\leq1$ has a linear extension
$g:A\!+\!\IR x\to\IR$ with $\|g\|\leq1$. In fact, a linear extension $g:A\!+\!\IR x\to\IR$ 
of $f$ satisfies $\|g\|\leq1$ if and only if
\[\sup_{y\in A}(f(y)-\|x-y\|)\leq g(x)\leq\inf_{y\in A}(f(y)+\|x-y\|).\]
In particular, both bounds exist and fall into the interval $[-\|x\|,\|x\|]$.
\end{theorem}

Since computing a value in between two numbers that are 
given as a supremum or infimum, respectively, is exactly
what the intermediate value theorem enables us to do,
we can conclude that the one-step Hahn-Banach theorem ($\HBT^1$)
is reducible to the intermediate value theorem ($\IVT$).

\begin{proposition}
\label{prop:HBT1-IVT}
$\HBT_X^1\leqW\IVT$ for every computable normed space $X$.
\end{proposition}

Our second main result that we prove in Section~\ref{sec:one-step}
is that also the one-step theorem exhibits its most complex behavior 
already on the space $X=\ell^1$.

\begin{theorem}
	$\HBT_{\ell^1}^1\equivW\IVT$.
\end{theorem}

The historical proofs of the Hahn-Banach theorem, provided by Banach and Hahn,
proceed from the one-step version
inductively using transfinite induction~\cite{NB11}.
Nowadays, the text book proofs use Zorn's lemma instead and in both cases
the proofs are based on the axiom of choice.
In the case of separable spaces, it is well-known~\cite{NB11} that the axiom
of dependent choice is sufficient to obtain the Hahn-Banach theorem inductively
from the one-step version.
We can mimic this proof and in this way we obtain a new proof of the upper bound.

\begin{proposition}
	$\HBT_X\leqW\IVT^\infty$ for every computable normed space $X$.
\end{proposition}

Here $\IVT^\infty$ denotes the infinite product of $\IVT$ in the sense of~\cite{Bra25},
where it is also proved that $\WKL^\infty\equivW\WKL$ and hence
$\IVT^\infty\equivW\WKL$ follows. 
Kihara noticed that the infinite product operation corresponds to the axiom of dependent
choice in some logical setting.\footnote{Takayuki Kihara, {\em The infinite loop operation and the axiom of dependent choice}, presentation at CCA 2025, Kyoto, Japan, 24 September 2025.}

If the space $X$ under consideration is finite-dimensional, then a finite
number of loops is sufficient to obtain the Hahn-Banach theorem.
In fact, we obtain the following result.

\begin{proposition}
\label{prop:HBT-finite}
	$\HBT_X\leqW\IVT^{[n-1]}$ for every $n$--dimensional computable
	nor\-med space $X$ and $n\geq 1$.
\end{proposition}

Here $\IVT^{[n]}$ denotes the $n$--fold compositional product of $\IVT$.
We note that $\IVT\lW\IVT^{[2]}$~\cite[Theorem~9.3]{BLRMP19}.
The intermediate value theorem $\IVT$ is known to have computable solutions
for computable instances.
This gives us immediately a non-uniform computable version of the
Hahn-Banach theorem as a corollary.

\begin{corollary}[Metakides and Nerode 1982]
\label{cor:MN82}
Let $X$ be a finite-dimensional computable normed space with a
computably separable subspace $A\In X$. Then every computable linear
functional $f:A\to\IR$ has a computable linear extension $g:X\to\IR$
with $\|g\|=\|f\|$.
\end{corollary}

Finally, we note that the upper bound given in Proposition~\ref{prop:HBT-finite}
is not tight. In fact, we prove the following result 
in Section~\ref{sec:ell1-2} for the two-dimensional case.

\begin{theorem}
	$\HBT_{\ell^1_2}\equivW\LLPO$.
\end{theorem}

The diagram in Figure~\ref{fig:HBT} shows the relevant Weihrauch degrees
that we are going to study.

\begin{figure}[htb]
	\begin{center}
		\begin{tikzpicture}[scale=0.9,auto=left,thick,every node/.style={fill=blue!20}]
            \node (LLPO) at (0,0) {$\HBT_{\ell_2^1}\equivW\HBT_{\ell_2^\infty}\equivW\LLPO$};
            \node (IVT) at (0,1) {$\HBT_{\ell^1}^1\equivW\IVT\equivW\C\C_{[0,1]}$};
            \node (WKL) at (0,2) {$\HBT_{\ell^1}\equivW\WKL\equivW\IVT^\infty$};    
			\draw[->] (WKL) edge (IVT);
			\draw[->] (IVT) edge (LLPO);			
		\end{tikzpicture}
		\caption{The Hahn-Banach theorem in the Weihrauch lattice.}
		\label{fig:HBT}
	\end{center}
\end{figure}
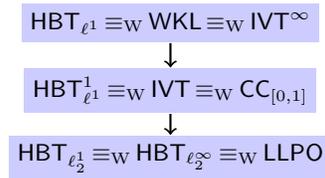

\section{Weihrauch Complexity and the Hahn-Banach Theorem}
\label{sec:Weihrauch}

We introduce concepts from computable analysis and Weihrauch complexity that we are
going to use and we refer the reader to \cite{BH21,Wei00} for more details.
We recall that a {\em representation} of a space $X$ is a surjective partial map $\delta_X:\In\IN^\IN\to X$.
In this case $(X,\delta_X)$ is called a {\em represented space}. If we have two represented spaces
$(X,\delta_X)$ and $(Y,\delta_Y)$, then we automatically have a representation $\delta_{\CC(X,Y)}$ of the space
of functions $f:X\to Y$ that have a continuous realizer. A partial function $F:\In\IN^\IN\to\IN^\IN$
is called a {\em realizer} of some partial multivalued function $f:\In X\mto Y$, if
\[\delta_YF(p)\in f\delta_X(p)\]
for all $p\in\dom(f\delta_X)$. In this situation we also write $F\vdash f$.
It is well-known that there are universal functions $\U:\In\IN^\IN\to\IN^\IN$ such that
for every continuous $F:\In\IN^\IN\to\IN^\IN$ there is some $q\in\IN^\IN$ such that
$F(p)=\U\langle q,p\rangle$ for all $p\in\dom(F)$. Here $\langle\cdot\rangle$ denotes some
standard pairing function on Baire space $\IN^\IN$ (we use this notation for pairs as well as for the pairing of sequences).
For short we write $\U_q(p):=\U\langle q,p\rangle$ for all $q,p\in\IN^\IN$.
Now we obtain a representation
$\delta_{\CC(X,Y)}$ of the set $\CC(X,Y)$ of total singlevalued functions $f:X\to Y$ with continuous realizers by
\[\delta_{\CC(X,Y)}(q)=f:\iff\U_q\vdash f.\]
It is well-known that for admissibly represented $T_0$--spaces $X,Y$ the function space $\CC(X,Y)$ consists
exactly of the usual continuous functions (see \cite{BH21,Wei00} for more details).
We need a representation of function spaces $\CC(A,Y)$ for
varying domains $A$. We use coproducts of the following form for this purpose.

\begin{definition}[Coproduct function spaces]
\rm
Let $X,Y$ be represented spaces and let
$(\PP(X),\delta_\PP)$ be a represented space with $\PP(X)\In 2^X$.
Then 
\begin{eqnarray*}
&\CC_\PP(X,Y):=\bigsqcup_{A\in\PP(X)}\CC(A,Y):=\{(f,A):A\in\PP(X)\mbox{ and }f\in\CC(A,Y)\}\
\end{eqnarray*}
denotes the {\em coproduct function space} that we represent by $\delta_{\CC_\PP}$,
defined by
\[\delta_{\CC_\PP}\langle q,p\rangle=(f,A):\iff \delta_\PP(p)=A\mbox{ and }\U_q\vdash f\]
for all total single-valued functions $f:A\to Y$ in $\CC(A,Y)$ with $A\in\PP(X)$.
\end{definition}

If $Y=\IR$, then we write for short $\CC_\PP(X):=\CC_\PP(X,\IR)$.
We will use for $\PP(X)$ for certain spaces of closed subsets of 
a computable metric space $X$. We recall that a {\em computable metric space}
$X$ is a metric space $(X,d)$ together with a dense sequence $s:\IN\to X$
such that $d\circ(s\times s)$ is computable. 
By $\SS(X)$ we denote the space of non-empty closed subsets $A\In X$
represented via a sequence $(x_n)_{n\in\IN}$ such that 
\[A=\overline{\{x_n:n\in\IN\}}.\]
By $\LL(X)$ we denote the space of non-empty closed subsets $A\In X$
represented via their {\em distance functions}
\[d_A:X\to\IR,x\mapsto\inf_{y\in A}d(x,y).\]
In the first case the name of a set $A$ is a name for a point in $X^\IN$,
in the second case a name for a function in $\CC(X)=\CC(X,\IR)$.
The computable sets $A\in\SS(X)$ are called {\em computably separable}
and the computable sets $A\in\LL(X)$ are called {\em located}~\cite{BP03}.
It follows from~\cite[Theorems~3.8, 3.9]{BP03} that the representation of 
$\LL(X)$ contains significantly more information on closed subspaces than that
of $\SS(X)$. In particular, the identity $\id:\LL(X)\to\SS(X)$ is computable
for every computable metric space, but typically the inverse is not 
computable (not even for $X=\IR$). We write $\CC_\SS$ or $\CC_\LL$ instead of $\CC_\PP$
if we use $\PP(X)=\SS(X)$ or $\PP(X)=\LL(X)$, respectively.

We recall that a {\em computable normed space} $X$ is a normed
space $(X,\|\cdot\|)$ together with a {\em fundamental sequence}
$e:\IN\to X$ (i.e., a sequence whose linear span is dense) 
such that the induced metric space is a computable metric space.
This metric space is defined via some standard enumeration $s:\IN\to X$
of the rational linear combinations of $e$. 
We only consider normed spaces over the field of real numbers $\IR$.

It is well-known that linear maps on separable normed spaces can be represented by their
values on a fundamental sequence together with a bound (see, e.g., \cite[Theorem~4.3]{Bra03c}).

\begin{proposition}[Linear bounded functionals]
\label{prop:linear}
Let $X$ be a computable nor\-med space with fundamental sequence $e:\IN\to X$.
The multivalued function
\[F:\In\CC(X)\mto\IR^\IN\times\IN,f\mapsto\{((f(e_n))_{n\in\IN},M):\|f\|\leq M\},\]
defined on all linear bounded functionals $f$, is computable and has a computable left inverse.
\end{proposition}

We can now define the {\em Hahn-Banach problem} 
$\HBT_X$ and the {\em one-step Hahn-Banach problem} $\HBT_X^1$.

\begin{definition}[Hahn-Banach problem]
\rm
Let $X$ be a computable normed space. Then we define:
\begin{enumerate}
\item $\HBT_X:\In\CC_\SS(X)\mto\CC(X)$, where $\HBT_X(f,A)$ is the set
\[\{g\mid\text{$g:X\to\IR$ is a linear extension of $f$ with $\|g\|\leq\|f\|$}\}\]
with $\dom(\HBT_X):=\{(f,A)\mid\text{$f:A\to\IR$ linear with $0<\|f\|<\infty$}\}$.\vspace{0.2cm}
\item $\HBT_X^1:\In\CC_\SS(X)\times X\mto\CC_\SS(X)$, where $\HBT_X^1(f,A,x)$ is the set
\[\{(g,A\!+\!\IR x)\mid\text{$g:A\!+\!\IR x\to\IR$ is a linear extension of $f$ with
	$\|g\|\leq\|f\|$}\}\]
with $\dom(\HBT_X^1):=\{(f,A,x)\mid\text{$f:A\to\IR$ linear with $0<\|f\|<\infty$}\}$.
\end{enumerate}
\end{definition}

We note that we do not require $x\not\in A$ for $\HBT_X^1$.
We use the usual concept of Weihrauch reducibility (see \cite{BGP21} for a survey) in order to compare problems. In general a {\em problem} is a multivalued function $f:\In X\mto Y$ on represented spaces $X,Y$. 
Here $\id:\IN^\IN\to\IN^\IN$ denotes the identity on Baire space.

\begin{definition}[Weihrauch reducibility]
	\rm
	Let $f:\In X\mto Y$ and $g:\In Z\mto W$ be problems. We say that
$f$ is {\em Weihrauch reducible} to $g$, in symbols $f\leqW g$, if there are computable
		$H,K:\In\IN^\IN\to\IN^\IN$ such that $H\langle\id,GK\rangle\vdash f$, whenever $G\vdash g$ holds. 
\end{definition}

As usual, we denote the corresponding equivalence by $\equivW$. 
We will also need a number of well-known benchmark problems that we
will use to characterize the above Hahn-Banach problems.
The problems in the following definition are known 
as {\em Weak K\H{o}nig's lemma} ($\WKL$), 
as {\em separation problem} ($\SEP$),
as {\em intermediate value theorem} ($\IVT$),
as {\em connected choice problem} ($\C\C_{[0,1]}$) of $[0,1]$
and as {\em lesser limited principle of omniscience}.
We write $\Tr$ for the set of {\em binary trees} $T\In\{0,1\}^*$ represented
via their characteristic functions. By $\range(f)=\{f(x):x\in X\}$
we denote the {\em range} of a function $f:X\to Y$.

\begin{definition}[Benchmark problems]
\rm
We define the following problems:
\begin{enumerate}
\itemsep 0.2cm
	\item $\WKL:\In\Tr\mto2^\IN,T\mapsto[T]$, where $\dom(\WKL)$ is the set of all infinite binary trees and $[T]$ denotes the set of infinite paths of $T$.
	\item $\SEP:\In\IN^\IN\times\IN^\IN\mto2^\IN,(p,q)\mapsto\{A:(\forall n)\;\range(p)\In A\In\IN\setminus\range(q)\}$, where
	$\dom(\SEP):=\{(p,q):\range(p)\cap\range(q)=\varnothing\}$.
	\item $\IVT:\In\CC[0,1]\mto[0,1],f\mapsto f^{-1}\{0\}$ 
          with $\dom(\IVT)=\{f:f(0)\cdot f(1)<0\}$.
	\item $\C\C_{[0,1]}:\In[0,1]^\IN\times[0,1]^\IN\mto[0,1],((a_n),(b_n))\mapsto\{x:(\forall n)\;a_n\leq x\leq b_n\}$\\ 
	with $\dom(\C\C_{[0,1]})=\{((a_n),(b_n)):(\forall n)\;a_n\leq a_{n+1}\leq b_{n+1}\leq b_n\}$.
	\item $\LLPO:\In2^\IN\times2^\IN\mto\{0,1\},(p_0,p_1)\mapsto\{i\in\{0,1\}:p_i=\widehat{0}\}$\\
	with $\dom(\LLPO):=\{(p_0,p_1):(\exists i\in\{0,1\})\;p_i=\widehat{0}\}$.
\end{enumerate}
\end{definition}

Here $\widehat{0}\in 2^\IN$ denotes the constant zero sequence.
The following equivalences are well-known.
The equivalence of Weak K\H{o}nig's lemma and the separation 
problem was proved by Gherardi and Marcone~\cite{GM09},
the equivalence of connected choice and the intermediate value
theorem is due to Gherardi and the first author~\cite[Proposition~3.6, Theorem~6.2]{BG11a}.
More results on connected choice can be found in~\cite{BLRMP19}.
The proof that Weak K\H{o}nig's lemma is closed under infinite
loops can be found in~\cite{Bra25}, as well as all other required results for infinite loops.

\begin{proposition}
	\label{prop:WKL-IVT}
	$\WKL^\infty\equivW\WKL\equivW\SEP\equivW\IVT^\infty$
	and $\IVT\equivW\C\C_{[0,1]}$. 
\end{proposition}

We now recall the definition of {\em infinite loops}, which 
were introduced in~\cite{Bra25}. Intuitively, $f^\infty=...\star f\star f$ can be
seen as an infinite compositional product. The {\em compositional product} of two problems $f,g$ on Baire space can
be defined by
$f\star g:=\langle\id\times f\rangle\circ\U\circ\langle\id\times g\rangle$.
By $f^{[n]}$ we denote the $n$--fold compositional product of $f$ with itself.
The infinite loop is an inverse limit construction based on this product.
Here $f(A)=\bigcup_{x\in A}f(x)$ for $f:\In X\mto Y$ and $A\In X$.

\begin{definition}[Infinite loop]
\rm
Let $f:\In\IN^\IN\mto\IN^\IN$ be a problem. Then we define
the {\em inverse limit} $f^\infty:\In\IN^\IN\mto\IN^\IN$ of $f$
by
\[f^\infty(q_0):=\{\langle q_0,q_1,q_2,...\rangle\in\IN^\IN:(\forall i)\;q_{i+1}\in \U\circ\langle \id\times f\rangle(q_i)\}\]
where $\dom(f^\infty)$ consists of all $q_0\in\IN^\IN$ such that
$A_0:=\{q_0\}\In\dom(\U\circ\langle\id\times f\rangle)$ and $A_{i+1}:=\U\circ\langle \id\times f\rangle(A_i)\In\dom(\U\circ\langle\id\times f\rangle)$ for all $i\in\IN$.
\end{definition}

This definition can be extended from problems on Baire space
to problems on arbitrary represented spaces using standard techniques.
It has also been proved in~\cite{Bra25} that $f\mapsto f^\infty$ is a monotone operation
with respect to (strong) Weihrauch reducibility.
Now we are well prepared to prove our main results.

\section{The One-Step Hahn-Banach Theorem}
\label{sec:one-step}

The way we have defined $\C\C_{[0,1]}$ makes
Proposition~\ref{prop:HBT1-IVT} a direct corollary
of Proposition~\ref{prop:WKL-IVT} and Theorem~\ref{thm:one-step}.
That is, the one-step Hahn-Banach theorem is reducible to the
intermediate value theorem for every computable normed space $X$.

\begin{proposition}
\label{prop:HBT1-CC}
$\HBT_X^1\leqW\C\C_{[0,1]}$ for every computable normed space $X$.
\end{proposition}
\begin{proof}
Given a functional $f:A\to\IR$ with $0<\|f\|<\infty$ and $x\in X$ we
can assume that $\|f\|\leq1$, because we can just divide $f$ by
an upper bound $M$ of $\|f\|$ that can be computed according to Proposition~\ref{prop:linear}. 
Then we obtain an extension $g:A+\IR x\to\IR$ with $\|g\|\leq1$, 
where we determine $g(x)$ with the help of $\C\C_{[0,1]}$ and Theorem~\ref{thm:one-step}.
By Proposition~\ref{prop:linear} we can actually compute $g$ as a point in the function space $\CC(A+\IR x)$
with the available information. 
In order to convert $g$ into an extension of the original functional,
we have to multiply it with $M$ again.
\qed
\end{proof}

We emphasize that this proof only yields an ordinary Weihrauch reduction,
not a strong one. 
If we apply the one-step version of the Hahn-Banach problem
repeatedly in an infinite loop, then we get a new proof of the well-known
reduction to Weak K\H{o}nig's lemma.

\begin{theorem}
$\HBT_X\leqW\WKL$ for every computable normed space. 
\end{theorem}
\begin{proof}
Starting from a functional $f:A\to\IR$ for which we can again
assume $\|f\|\leq1$, we can just repeatedly
apply $\HBT_X^1$ for the fundamental sequence $(e_n)_n$ of the space $X$.
Inductively, starting from $f_0:=f$ and $A_0:=A$ 
this yields functionals $h_{n+1}:A_n+\IR e_n\to\IR$ and closed sets $A_{n+1}$
as closure of the linear span of $A_n+\IR e_n$.
Using Proposition~\ref{prop:linear} we can compute extensions of each $h_{n+1}$
to a functional of type $f_{n+1}:A_{n+1}\to\IR$ that is used for the next
application of $\HBT_X^1$. Hence, every $f_{n+1}$ is a linear extension of $f_n$
with $\|f_n\|\leq1$.
Altogether, the values $(f_{n+1}(e_n))_{n\in\IN}$ determine
a linear functional $g:X\to\IR$ that extends $f$ with $\|g\|\leq1$
and these data suffice to obtain $g$ as a point in $\CC(X)$ by Proposition~\ref{prop:linear}.
Altogether, by Proposition~\ref{prop:WKL-IVT}, this proves 
$\HBT_X\leqW\IVT^\infty\equivW\WKL$. 
\qed
\end{proof}

In the finite-dimensional case, the same argument requires only finitely
many applications of $\IVT$, which
yields Proposition~\ref{prop:HBT-finite}.
In the next section we will see that this bound is not sharp,
not even for the $\ell^1$--norm on $\IR^2$.

Next we want to prove that the one-step Hahn-Banach theorem
reaches its maximal complexity for $X=\ell^1$.
We use the computable linear isometry
\begin{eqnarray}
	R:\ell^1(\IN\times\{0,1\})\to\ell^1,R((x_{n,i})_{(n,i)\in\IN\times\{0,1\}})(2n+i):=x_{n,i},
\label{eq:R}
\end{eqnarray}
which allows us to identify $\ell^1(\IN\times\{0,1\})$ with $\ell^1$.
We also use the standard fundamental sequence $(e_{n,i})$ of unit vectors of $\ell^1(\IN\times\{0,1\})$.

\begin{proposition}
	$\C\C_{[0,1]}\leqW\HBT_{\ell^1}^1$.
\end{proposition}
\begin{proof}
Given two sequences $(a_n)_{n\in\IN}$ and $(b_n)_{n\in\IN}$ of rational numbers
in $[0,1]$ such that $a_n\leq a_{n+1}$ and $b_{n+1}\leq b_{n}$ with
$a:=\sup_{n\in\IN}a_n\leq\inf_{n\in\IN}b_n=:b$, the goal is to find a real number $y\in[0,1]$
with $a\leq y\leq b$. Without loss of generality, we can even assume $a_n<b_n$ for all $n\in\IN$.

We work with the space $X=\ell^1(\IN\times\{0,1\})$ that is isomorphic to $\ell^1$ by (\ref{eq:R}).
We now compute a functional $f:A\to\IR$ on a subspace $A\In X$
with $\|f\|\leq1$ and a point $x\in X$ such that every 
linear extension $g:A+\IR x\to\IR$ of $f$ with $\|g\|\leq1$
satisfies $a\leq g(x)\leq b$. This proves $\C\C_{[0,1]}\leqW\HBT_{\ell^1}^1$.

In order to construct $f$, we first compute
\[\alpha_n:=\frac{a_n-b_n}{2}<0\text{ and }\beta_n:=\frac{a_n+b_n}{2}\in[0,1]\]
from the input data and then for all $n\in\IN$
\[u_n:=e_{n,0}+\alpha_n e_{n,1}\text{ and }v_n:=e_{n,0}-e_{n+1,0}.\]
Hence, we can also compute the closure $A\in\SS(X)$
of the linear span of $B:=\{v_n,u_n:n\in\IN\}$
and $x:=\sum_{n\in\IN}2^{-n-1}e_{n,0}\in X$. Since $B$ is linearly independent, 
there is a unique linear $f_0:\span(B)\to\IR$ with the values 
\[f_0(v_n):=0\text{ and }f_0(u_n):=\beta_n\]
for all $n\in\IN$. We claim that $f_0$ is bounded with $\|f_0\|\leq1$
and hence it extends uniquely to a linear bounded
functional $f:A\to\IR$ with $\|f\|\leq1$ by the Hahn-Banach theorem.
We continue assuming this claim for the moment.

Let $h:A+\IR x\to\IR$ be a functional that we receive as
output of $\HBT^1_X(f,A,x)$. By the classical Hahn-Banach theorem
$h$ has an extension $g:X\to\IR$ that is  
a linear continuous extension of $f$ with
$\|g\|\leq1$. Since $Y=\ell^\infty(\IN\times\{0,1\})$ is the dual space
of $X$, there is a $w\in Y$ 
with $\|w\|_\infty\leq1$ such that
\[g(z)=\langle w,z\rangle:=\sum\nolimits_{k\in\IN}(w_{k,0}z_{k,0}+w_{k,1}z_{k,1})\]
for all $z\in X$ (here $\langle\cdot\rangle$ simply denotes the duality pairing).
Then 
\[0=g(v_n)=\langle w,e_{n,0}-e_{n+1,0}\rangle=w_{n,0}-w_{n+1,0}\]
and hence the values $y:=w_{n,0}$ are constant for all $n\in\IN$.
We also obtain
\[\beta_n=g(u_n)=\langle w,e_{n,0}+\alpha_ne_{n,1}\rangle=w_{n,0}+\alpha_nw_{n,1}=y+\alpha_nw_{n,1}\]
for all $n\in\IN$. Since $\|w\|_\infty\leq1$, we have $|w_{n,1}|\leq 1$ and hence
\[y=\beta_n-\alpha_nw_{n,1}\in[\beta_n+\alpha_n,\beta_n-\alpha_n]=[a_n,b_n]\]
for all $n\in\IN$, which implies $a\leq y\leq b$.

The preceding argument can also be reversed.
If we start with some arbitrary $y$ with $a\leq y\leq b$,
then we can choose $w_{n,1}$ with $|w_{n,1}|\leq1$ such 
that $\beta_n=y+\alpha_nw_{n,1}$ and $w_{n,0}=y$.
Then $w=(w_{n,0},w_{n,1})_n\in Y$ is a point with $\|w\|_\infty\leq1$
that hence defines a functional $g:X\to\IR,z\mapsto\langle w,z\rangle$ with $\|g\|\leq1$
and this functional extends $f_0$ by the same calculation as above.
This proves the claim that $f_0$ can be extended to $f:A\to\IR$ with $\|f\|\leq1$.

If $g:X\to\IR$ is now an extension of $f$ with $\|g\|\leq1$ as above, 
then we can evaluate $g$ on $x$
and we obtain
\[g(x)=\left\langle w,\sum_{n\in\IN}2^{-n-1}e_{n,0}\right\rangle=\sum_{n\in\IN}2^{-n-1}w_{n,0}=\sum_{n\in\IN}2^{-n-1}y=y\in[a,b].\]
This implies $h(x)=y\in[a,b]$ and completes the proof.
\QED
\end{proof}

Together with Proposition~\ref{prop:HBT1-CC} we obtain the desired characterization.

\begin{corollary}
$\HBT_{\ell^1}^1\equivW\IVT\equivW\C\C_{[0,1]}$.
\end{corollary}

\section{The Hahn-Banach Theorem for $\ell^1_2$}
\label{sec:ell1-2}

We identify the space $\ell^p_2$ with $\IR^2$ equipped with the $\ell^p$--norm.
In this section we want to prove that for $\ell^1_2$ (and $\ell^\infty_2$), 
the Hahn-Banach theorem is equivalent to $\LLPO$,
which shows that the upper bound given in Proposition~\ref{prop:HBT-finite}
is not tight, not even for the $\ell^1$--norm.

In fact, it suffices to consider the case of $\ell^1$, as there is a computable
linear isometric map
\begin{eqnarray}
	S:\IR^2\to\IR^2,(u,v)\mapsto\left(\frac{u-v}{2},\frac{u+v}{2}\right)
\label{eq:S}
\end{eqnarray}
that satisfies 
$
\|S(u,v)\|_1=\frac12(|u+v|+|u-v|) = \max(|u|,|v|)=\|(u,v)\|_\infty.
$
In Figure~\ref{fig:balls} the respective unit balls are illustrated.
We consider the case $X=\ell^1_2$. If we have a functional $f:A\to\IR$
with $\|f\|=1$ defined on a one-dimensional subspace $A\In\IR^2$, then the extension
of this functional to a functional $g:\IR^2\to\IR$ with $\|g\|=1$ is actually
uniquely determined, provided that $A$ does not cross any corner of the unit ball.
This is because $f^{-1}\{1\}$ is an affine subspace that is not allowed
to run through the interior of the ball (because $\|f\|=1$) and 
hence it has to include one of the sides of the unit ball, which fixes
all the values of the extension. One can use $\LLPO$ to determine on
which side of the unit ball the affine hyperplane $f^{-1}\{1\}$ lies.
We recall that $\LLPO$ is equivalent to the problem of determining
one of the cases $r\leq0$ or $r\geq0$, which holds for a real $r\in\IR$.
For the other direction of the reduction we use an idea of Ishihara~\cite{Ish90}.

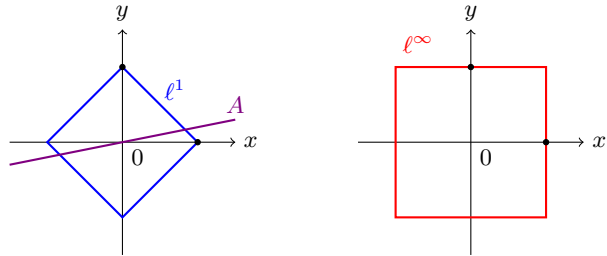
\begin{figure}[htb]
\begin{center}
\begin{tikzpicture}
	\draw[->] (-1.5,0) -- (1.5,0) node[right] {$x$};
	\draw[->] (0,-1.5) -- (0,1.5) node[above] {$y$};
	\node at (0.2,-0.2) {$0$};
	\draw[blue, thick] (1,0) -- (0,1) -- (-1,0) -- (0,-1) -- cycle;
	\node[blue] at (0.7,0.7) {$\ell^1$};
	\filldraw[black] (1,0) circle (1pt);
	\filldraw[black] (0,1) circle (1pt);
	\draw[violet, thick] (-1.5,-0.3) -- (0,0) -- (1.5,0.3);
	\node[violet] at (1.5,0.5) {$A$};
	\end{tikzpicture}
\hspace{1cm}
\begin{tikzpicture}
	\draw[->] (-1.5,0) -- (1.5,0) node[right] {$x$};
	\draw[->] (0,-1.5) -- (0,1.5) node[above] {$y$};
	\node at (0.2,-0.2) {$0$};
	\draw[red, thick] (1,1) -- (-1,1) -- (-1,-1) -- (1,-1) -- cycle;
	\node[red] at (-0.7,1.3) {$\ell^\infty$};
	\filldraw[black] (1,0) circle (1pt);
	\filldraw[black] (0,1) circle (1pt);
\end{tikzpicture}
\end{center}
\vspace{-0.4cm}
\caption{Unit balls in $\IR^2$ with respect to $\ell^1$ and $\ell^\infty$.}
\label{fig:balls}
\end{figure}

\begin{proposition}
$\HBT_{\ell^1_2}\equivW\HBT_{\ell^\infty_2}\equivW\LLPO$.
\end{proposition}
\begin{proof}
We consider the case of $X=\ell^1_2$, i.e., $\IR^2$ with the $\ell^1$--norm.
Given a functional $f:A\to\IR$ for some linear subspace $A\In\IR^2$
with $0<\|f\|\leq1$. In fact, we can assume $\|f\|=1$, as we can divide
$f$ by its norm (as the operator norm is computable for finite-dimensional spaces~\cite{Bra03c}). 
We know that $A\not=\{0\}$ since $\|f\|>0$. Hence, we can find
some $0\not=x=(x_0,x_1)\in A$ with $\|x\|_1=1$ and we can find some $i\in\{0,1\}$ with $x_i\not=0$.  
Without loss of generality, we assume $x_0> 0$ and $f(x_0,x_1)=1$. 
If also $x_1\not=0$, then the subspace $A$ is not a one-dimensional
subspace that crosses one of the corners of the unit ball and the extension
of $f$ is uniquely determined. 
In fact, if $x_1>0$, then the norm-preserving extension $g$ is uniquely
determined by the additional condition $g(0,1)=1$, and if $x_1<0$, then it is
uniquely determined by $g(0,-1)=1$. If $x_1=0$, then both of these
norm-preserving extensions are possible.
With the help of $\LLPO$ we can select one of these cases.
The remaining cases are handled analogously. 
Altogether, this shows $\HBT_{\ell^1_2}\leqW\LLPO$.

For the proof of $\LLPO\leqW\HBT_{\ell^1_2}$ we follow a construction
of Ishihara~\cite{Ish90}. Given $r\in\IR$ we consider
$x=(1,r)\in\IR^2$, the subspace $A:=\{ax:a\in\IR\}\In\IR^2$
and the functional $f:A\to\IR$ with $f(ax):=a\cdot\|x\|_1=a(1+|r|)$.
Then $\|f\|=1$. Let $g:\IR^2\to\IR$ be a linear extension of $f$ with
$\|g\|=1$. Then $|g(e_i)|\leq1$ holds for the two unit vectors $e_1=(1,0),e_2=(0,1)\in\IR^2$.
Hence
\begin{eqnarray}
	1+|r|=f(x)=g(x)=g(e_1+re_2)=g(e_1)+r g(e_2)\leq g(e_1)+|r|.
\end{eqnarray}
Then $g(e_1)=1$ and hence $rg(e_2)=|r|$. In order to check whether
$r\leq0$ or $r\geq0$ we just have to find out whether
$g(e_2)>-1$ or $g(e_2)<1$. These conditions are semi-decidable in 
the input and can hence be tested in parallel. Depending on which
one is witnessed first, we output $1$ or $0$, respectively.
Since $r>0$ implies $g(e_2)=1$ and $r<0$ implies $g(e_2)=-1$,
only one test can succeed in these cases. If $r=0$, both
tests can succeed.

The statement for $\ell^\infty_2$ follows using the computable isometry (\ref{eq:S})
\qed
\end{proof}

\section{The Hahn-Banach Theorem for $\ell^1$}
\label{sec:ell1}

For their proof of the reduction $\WKL\leqW\HBT$ Gherardi and Marcone~\cite{GM09}
followed the construction of Brown and Simpson~\cite{BS86a,Sim09},
who in turn used ideas of similar constructions of
 Bishop~\cite{Bis67}, Metakides, Nerode and Shore~\cite{MN82,MNS85}.
We briefly recall the construction due to Gherardi and Marcone.
To every instance $(p, q)\in\IN^\IN\times\IN^\IN$ of the separation problem, i.e., with
$\range(p)\cap\range(q)=\varnothing$, they associate a Banach space 
$(X_{p,q},\|\cdot\|_{p,q})$ that is defined as follows. Firstly,
\[
\delta_n :=
\begin{cases}
	2^{-k-1} & \text{if } k=\min\{i\in\IN:p(i)=n\} \text{ exists}\\
	-2^{-k-1} & \text{if }k=\min\{i\in\IN:q(i)=n\} \text{ exists}\\
	0 & \text{otherwise}.
\end{cases}
\]
and then $\varepsilon_n:=\frac{1-\delta_n}{1+\delta_n}$ for all $n\in\IN$.
Then one can obtain norms on $\IR^2$ by
\[
\|(\alpha,\beta)\|_{p,q,n} :=
\begin{cases}
	\max(|\varepsilon_n\alpha+\beta|,\ |\alpha-\beta|) & \text{if }\varepsilon_n<1,\\[1.2ex]
	\max(|\alpha+\beta|,\ |\varepsilon_n^{-1}\alpha-\beta|) & \text{if }\varepsilon_n>1,\\[1ex]
	\max(|\alpha+\beta|,\ |\alpha-\beta|) & \text{if }\varepsilon_n=1.
\end{cases}
\]
For $x=(\alpha_n,\beta_n)_{n\in\IN}\in(\IR^2)^\IN$ we use the notation
$x_n=(\alpha_n,\beta_n)$ and $x_{n,0}=\alpha_n$ and $x_{n,1}=\beta_n$.
Now one obtains a Banach space $(X_{p,q},\|\cdot\|_{p,q})$ with
	\[
X_{p,q}
:=\Bigl\{x\in(\IR^2)^\IN: \|x\|_{p,q}<\infty\Bigr\},
\text{ where }
\|x\|_{p,q}:=\sum_{n=0}^\infty 2^{-n-1}\|x_n\|_{p,q,n}.
\]
That is, $\|\cdot\|_{p,q}$ is a $\ell^1$--sum of weighted $\ell^\infty$--blocks in $(\IR^2,\|\cdot\|_{p,q,n})$.
On this Banach space Gherardi and Marcone considered the functional
\begin{eqnarray}
	f_{p,q}:A_{p,q}\to\IR,x\mapsto\sum_{n=0}^\infty 2^{-n-1}x_{n,0} 
\label{eq:functional}
\end{eqnarray}
for the subspace
\begin{eqnarray}
A_{p,q}:=\{x\in X_{p,q}: x_{n,1}=0\text{ for all $n\in\IN$}\}.
\label{eq:subspace}
\end{eqnarray}
Then $\|f_{p,q}\|=1$ and from a functional 
$g:X_{p,q}\to\IR$ that extends $f_{p,q}$ with $\|g\|=1$ one can compute
a set $B\In\IN$ that separates $\range(p)$ and $\range(q)$, as
\begin{eqnarray}
n\in\range(p)\TO g(z_n)=-2^{-n-1}\text{ and }n\in\range(q)\TO g(z_n)=+2^{-n-1}
\label{eq:B}
\end{eqnarray}
for all $n\in\IN$, where $z_n\in X_{p,q}$ is defined
by $(z_n)_n=(0,1)$ and $(z_n)_k=(0,0)$ for $k\not=n$.

In order to prove that the Hahn-Banach theorem exhibits its maximal
power for $\ell^1$, we construct a computable linear isometry.

\begin{proposition}
	There exists a linear isometry $F_{p,q}:X_{p,q}\to\ell^1$ that
	is computable uniformly in instances $(p,q)\in\IN^\IN\times\IN^\IN$ of the separation problem.
\end{proposition}
\begin{proof}
The linear map $T_n:\IR^2\to\IR^2$ with
\[
T_n(\alpha,\beta):=
\begin{cases}
	(\varepsilon_n\alpha+\beta,\ \alpha-\beta) & \text{if }\varepsilon_n<1\\
	(\alpha+\beta,\ \varepsilon_n^{-1}\alpha-\beta) & \text{if }\varepsilon_n>1\\
	(\alpha+\beta,\ \alpha-\beta) & \text{if }\varepsilon_n=1
\end{cases}
\]
is computable uniformly in $n$ relative to $p,q$ 
and satisfies 
\[\|(\alpha,\beta)\|_{p,q,n} \;=\; \|T_n(\alpha,\beta)\|_\infty.\]
Now we use the computable linear map $S:\IR^2\to\IR^2$ from (\ref{eq:S})
Hence, for $ST_n=S\circ T_n$ and $x\in\IR^2$ we obtain
$|(ST_n(x))_0|+|(ST_n(x))_1|=\|ST_n(x)\|_1=\|x\|_{p,q,n}$.
Now we can define $F_{p,q}:X_{p,q}\to\ell^1$
by
\[F_{p,q}(x)(2n+i):=(2^{-n-1}ST_n(x_n))_i\]
for all $x\in X_{p,q}$, $n\in\IN$ and $i\in\{0,1\}$. 
Finally, we obtain
\[\|F_{p,q}(x)\|_1=\sum_{n=0}^\infty2^{-n-1}\|ST_n(x_n)\|_1=\sum_{n=0}^\infty2^{-n-1}\|x_n\|_{p,q,n}=\|x\|_{p,q}\]
and $F_{p,q}$ is computable uniformly in $(p,q)$.
\qed
\end{proof}

Since $F_{p,q}:X_{p,q}\to\ell^1$ is an injective computable linear map on computable
Banach spaces, it has a computable inverse $F_{p,q}^{-1}:\range(F_{p,q})\to X_{p,q}$
by the computable version of the Banach Inverse Mapping Theorem~\cite[Corollary~5.3]{Bra09}.
However, this does not automatically hold uniformly in $p,q$. 
But since $F_{p,q}$ is even an isometry, the operator norm of the inverse is $1$
and hence we obtain uniformity in $p,q$ by~\cite[Theorem~5.9]{Bra09}.
This allows us to obtain the following conclusion.

\begin{proposition}
	\label{prop:ell1}
	$\SEP\leqW\HBT_{\ell^1}$
\end{proposition}
\begin{proof}
Given an instance $(p,q)\in\IN^\IN\times\IN^\IN$ of the separation
problem, we can compute the functional $f_{p,q}:A_{p,q}\to\IR$ from
(\ref{eq:functional})
and we obtain a functional $f:A\to\IR$ with $A:=F_{p,q}(A_{p,q})$
by $f:=f_{p,q}\circ F_{p,q}^{-1}$. Since everything is uniform in $p,q$,
we can compute $(f,A)\in\CC_\SS(\ell^1)$. 
Now we can apply $\HBT_{\ell^1}$ in order to obtain a linear extension $g:\ell^1\to\IR$
of $f$ with $\|g\|\leq\|f\|$.
Then $g':=g\circ F_{p,q}$ is a linear functional $g':X_{p,q}\to\IR$ that
extends $f_{p,q}$. Because $F_{p,q}$ is an isometry we have $\|g'\|=\|f_{p,q}\|=1$. Hence
we obtain the values $g'(z_n)=g\circ F_{p,q}(z_n)$ from which we can determine
a separating set $B\in\SEP(p,q)$ using (\ref{eq:B}).
\QED
\end{proof}

Now we obtain the following result.

\begin{corollary}
	$\HBT_{\ell^1}\equivW\WKL$.
\end{corollary}

\section{The Hahn-Banach Theorem for Located Subspaces}
\label{sec:located}

The proof of Proposition~\ref{prop:ell1} shifts the 
complexity from the space $X_{p,q}$ into the functional $f:A\to\IR$ and
the subspace $A\In\ell^1$ using the isometry $F_{p,q}$.
Hence, it is a relevant question whether the complexity of the Hahn-Banach 
theorem on $\ell^1$ can be reduced by providing more information on the
subspace $A$. We will prove that this is not the case, even if we provide
the subspace in form of its distance function $d_A$, i.e., if $\SS(\ell^1)$ is replaced
by $\LL(\ell^1)$ in the definition of $\HBT_{\ell^1}$.
We first prove that the subspaces from (\ref{eq:subspace}) can be computed
as points in $\LL(X_{p,q})$.
  
\begin{proposition}
	The sets $A_{p,q}\in\LL(X_{p,q})$ can be computed uniformly
	in instances $(p,q)\in\IN^\IN\times\IN^\IN$ of the separation problem.
\end{proposition}
\begin{proof}
	The definition of $\|\cdot\|_{p,q,n}$ implies that
	$\|(\alpha,\beta)\|_{p,q,n}\geq|\beta|=\|(0,\beta)\|_{p,q,n}$ 
	and $\|(\alpha,\beta)\|_{p,q,n}\geq\frac{1}{2}|\alpha|=\frac{1}{2}\|(\alpha,0)\|_{p,q,n}$
	for all $(\alpha,\beta)\in\IR^2$. 
	In particular, $(\alpha_n,\beta_n)_n\in X_{p,q}$ implies $(\alpha_n,0)_n\in X_{p,q}$. 
	Hence, the following
	infimum for $(\alpha_n,\beta_n)_n\in X_{p,q}$ is attained for $\alpha_n'=\alpha_n$ with the given value
	\[d_{A_{p,q}}((\alpha_n,\beta_n)_n)=\inf_{(\alpha_n',0)_n\in X_{p,q}}\sum_{n=0}^\infty2^{-n-1}\|(\alpha_n-\alpha_n',\beta_n)\|_{p,q,n}=\sum_{n=0}^\infty2^{-n-1}|\beta_n|.\]
    Thus, $d_{A_{p,q}}$ is computable in $p,q$, since
    $\sum_{n=0}^\infty2^{-n-1}|\beta_n|\leq\|(\alpha_n,\beta_n)\|_{p,q}$.
	\qed
\end{proof}

Next we prove that locatedness is preserved by computable linear isometries.

\begin{proposition}
	Let $F:X\to Y$ be a computable bijective linear isometry on computable normed
	spaces $(X,\|\cdot\|_X)$ and $(Y,\|\cdot\|_Y)$.
	Then 
	\[F:\LL(X)\to\LL(Y),A\mapsto F(A)\]
    is computable. This even holds uniformly in $F$.
\end{proposition}
\begin{proof}
	Since $F$ is a bijective linear isometry, we obtain
	\begin{eqnarray*}
		d_{F(A)}(x)&=&\inf_{z\in F(A)}\|x-z\|_Y=\inf_{y\in A}\|F(F^{-1}(x)-y)\|_Y\\
		&=&\inf_{y\in A}\|F^{-1}(x)-y\|_X=d_A(F^{-1}(x))
	\end{eqnarray*}
	for all $x\in Y$. If $F$ is given, then we can compute $F^{-1}$ by~\cite[Theorem~5.9]{Bra09} as $\|F^{-1}\|=1$. If, additionally,
	$A\in\LL(X)$ is given in form of $d_A\in\CC(X)$, then we can 
	compute $d_{F(A)}\in\CC(Y)$ by the equation above and hence
	$F(A)\in\LL(Y)$.
	\QED
\end{proof}

Now we can transfer the proof of Proposition~\ref{prop:ell1}
from $\SS(\ell^1)$ to $\LL(\ell^1)$.

\begin{corollary}
	$\HBT_{\ell^1}\equivW\WKL$, even if the space $\SS(\ell^1)$ in the definition
	of $\HBT_{\ell^1}$ is replaced by $\LL(\ell^1)$.
\end{corollary}

The following corollary strengthens Corollary~\ref{cor:MNS85}.

\begin{corollary}
	There exists a computable linear functional $f:A\to\IR$ on a located closed
	subspace $A\In\ell^1$ with $\|f\|=1$ and without a computable
	linear extension ${g:\ell^1\to\IR}$ with $\|g\|=1$. 
\end{corollary}

\begin{credits}
\subsubsection{\ackname} 
We acknowledge funding by the German Research Foundation (DFG, Deutsche Forschungsgemeinschaft) -- project number 554999067 and 
by the National Research Foundation of South Africa (NRF) -- grant number 151597.
\end{credits}

%
%
%
\bibliographystyle{splncs04}
\bibliography{C:/Users/\user/Documents/Spaces/Research/Bibliography/lit}

\end{document}